\documentclass[10pt]{amsart}
\usepackage{amssymb}
\usepackage{bm}
\usepackage{graphicx}
\usepackage[centertags]{amsmath}
\usepackage{amsfonts}
\usepackage{amsthm}
\newtheorem{thm}{Theorem}
\newtheorem{cor}[thm]{Corollary}
\newtheorem{lem}[thm]{Lemma}

\newtheorem{prop}[thm]{Proposition}

\newtheorem{defn}[thm]{Definition}
\theoremstyle{definition}
\newtheorem{rem}{Remark}

\newcommand{\rr}{\mathbb{R}}
\newcommand{\nn}{\mathbb{N}}
\newcommand{\ee}{\varepsilon}

\newcommand{\SB}{\mathbf{\Sigma}}
\newcommand{\PB}{\mathbf{\Pi}}

\newcommand{\supp}{\mathrm{supp}}

\begin{document}

\title{Dichotomies of the set of test measures of a Haar-null set}
\author{Pandelis Dodos}
\address{National Technical University of Athens, Faculty of Applied
Sciences, Department of Mathematics, Zografou Campus, 157 80,
Athens, Greece.}
\email{pdodos@math.ntua.gr}

\footnotetext[1]{Research supported by a grant of EPEAEK program ``Pythagoras".}

\maketitle

%------------------------Abstract-------------------------------%

\begin{abstract}
We prove that if $X$ is a Polish space and $F$ is a face of $P(X)$ with the Baire
property, then $F$ is either a meager or a co-meager subset of $P(X)$. As a
consequence we show that for every abelian Polish group $X$ and every analytic
Haar-null set $A\subseteq X$, the set of test measures $T(A)$ of $A$ is either
meager or co-meager. We characterize the non-locally-compact groups as the ones
for which there exists a closed Haar-null set $F\subseteq X$ with $T(F)$ is meager.
Moreover, we answer negatively a question of J. Mycielski by showing that for every
non-locally-compact abelian Polish group and every $\sigma$-compact subgroup $G$
of $X$ there exists a $G$-invariant $F_\sigma$ subset of $X$ which is neither
prevalent nor Haar-null.
\end{abstract}

%--------------------------Introduction-----------------------%

\section{Introduction and auxiliary lemmas}

A universally measurable subset $A$ of an abelian Polish group $X$ is called \textit{Haar-null}
if there exists a probability measure $\mu$ on $X$, called a \textit{test measure}
of $A$, such that $\mu(x+A)=0$ for every $x\in X$. This definition is due to J. P. R.
Christensen \cite{C} and extends the notion of a Haar-measure zero set. The same
class of sets has been also considered independently by B. R. Hunt, T. Sauer and
J. A. Yorke in \cite{HSY}. They used the term \textit{shy} instead of Haar-null.
The complements of Haar-null sets are called \textit{prevalent}. In Christensen's
paper \cite{C} a number of important properties of Haar-null sets were established.
In particular, he showed that the class of Haar-null sets is a $\sigma$-ideal, which
in the case of non-locally-compact groups contains all compact sets. On the other
hand, he proved that if $X$ is locally-compact, then Haar-null sets are precisely
the Haar-measure zero sets and so his definition is indeed a generalization.
There are, however, a number of properties of Haar-null sets which differentiate
the locally from the non-locally compact case (see \cite{BL} and the references therein).
An important example is the countable chain condition, which is not satisfied
if $X$ is non-locally-compact (this is due to R. Dougherty for a large class
of abelian groups \cite{Dou} and to S. Solecki in general \cite{S1}).

Note that the test measure of a Haar-null set is not unique. Our motivation for
this paper was to investigate the structure of the set of test measures $T(A)$ of
a Haar-null set $A$. One natural question is about its size. Our first result states
that there are only two extreme possibilities.
\medskip

\noindent \textbf{Theorem A.} \textit{Let $X$ be an abelian Polish group. Then for
every analytic Haar-null set $A\subseteq X$, the set
\[ T(A)=\{\mu\in P(X): \mu(x+A)=0 \text{ for every } x\in X\}\]
is either a meager or a co-meager subset of $P(X)$.}
\medskip

Actually, Theorem A is a consequence of a general property shared by the faces of
$P(X)$. Specifically, we prove the following zero-one law.
\medskip

\noindent \textbf{Theorem B.} \textit{Let $X$ be a Polish space and $F$ a face of
$P(X)$ with the Baire property. Then $F$ is either a meager or a co-meager subset
of $P(X)$.}
\medskip

\noindent The crucial point in the proof of Theorem B is the fact that convex averaging
is open in $P(X)$.

Theorem A justifies the following definition. An analytic Haar-null set $A$ is called
\textit{strongly} Haar-null if $T(A)$ is co-meager. Otherwise it is called \textit{weakly}
Haar-null. In every abelian Polish group there exist strongly and weakly Haar-null
sets. However the existence of a closed weakly Haar-null set characterizes the
non-locally-compact groups.
\medskip

\noindent \textbf{Theorem C.} \textit{Let $X$ be an abelian Polish group. Then $X$
is non-locally-compact if and only if it contains a closed weakly Haar-null set.}
\medskip

\noindent The proof of the above theorem is descriptive set-theoretic and uses the
results of S. Solecki in \cite{S2}.

We also deal with a problem of J. Mycielski. He asked in \cite{My} the following.
If $X$ is a non-locally-compact abelian Polish group, $G$ a countable dense subgroup
of $X$ and $A\subseteq X$ is a $G$-invariant universally measurable set, then is it
true that $A$ is either prevalent or Haar-null? The problem was answered negatively
by R. Dougherty in $\rr^\nn$ (see \cite{Dou}). Using a result of E. Matou\v{s}kov\'{a}
and M. Zelen\'{y} we show that for every non-locally-compact abelian Polish group
$X$ and every $\sigma$-compact subgroup $G$ of $X$, there exists a $G$-invariant
$F_\sigma$ subset of $X$ which is neither prevalent nor Haar-null.

Finally, a measure-theoretic analogue of Theorem A is given on the last section.
It is based on the fact that $T(A)$ is invariant under the group of homeomorphisms
$\mu\to \mu* \delta_x$, where $x\in X$.
\medskip

\noindent \textsc{Notation.} In what follows $X$ will be a Polish space
(additional assumptions will be stated explicitly). By $d$ we denote a compatible
complete metric of $X$. If $x\in X$ and $r>0$, then by $B(x,r)$ we denote the set
$\{y\in X: d(x,y)<r\}$. By $P(X)$ we denote the space of all Borel probability measures
on $X$, while by $M_+(X)$ the space of all positive finite Borel measures. Then
both $P(X)$ and $M_+(X)$ equipped with the weak topology become Polish spaces
(see, for instance, \cite{Pa}). If $d$ is a compatible complete metric
of $X$, then a compatible complete metric of $P(X)$ is the so-called L\'{e}vy
metric $\varrho$, defined by
\[ \rho(\mu,\nu)=\inf \{\delta\geq 0: \mu(A)\leq \nu(A_\delta)+\delta
 \text{ and } \nu(A)\leq \mu(A_\delta)+\delta\} \]
where $A_\delta=\{x\in X:d(x,A)\leq\delta\}$. All balls in $P(X)$ are taken with
respect to the L\'{e}vy metric $\rho$. That is, if $\mu\in P(X)$ and $r>0$,
then by $B(\mu,r)$ we denote the set $\{\nu\in P(X):\rho(\mu,\nu)<r\}$ (from the
context it will be clear whether we refer to a ball in $P(X)$ or in $X$).
Finally, for every $\mu\in P(X)$, by $\supp\mu$ we denote the support of the measure
$\mu$. All the other pieces of notation we use are standard (for more information
we refer to \cite{Kechris}).

%----------------------------------------------------------%

\section{Faces of $P(X)$}

Through this section $X$ will be a Polish space. As usual, we say that $F\subseteq P(X)$
is a face of $P(X)$ if $F$ is an extreme convex subset of $P(X)$. For every $t\in (0,1)$
consider the function $T_t:P(X)\times P(X)\to P(X)$, defined by
\[ T_t(\mu,\nu)=t\mu+(1-t)\nu. \]
Clearly every $T_t$ is continuous. The following lemma provides an estimate for their
range.
\begin{lem} \label{l1}
Let $r>0$ and $m,\mu\in P(X)$ be such that $\rho(m,\mu)<r$. Let also $t\in (0,1)$.
Then for every $\nu\in P(X)$ we have that $\rho\big(T_t(\nu,m),\mu)\leq t+r$.
\end{lem}
\begin{proof}
Let $A\subseteq X$ Borel. Then note that
\begin{eqnarray*}
t\nu(A)+(1-t)m(A)\leq t+m(A) & \leq & t+\mu(A_r)+r\\
& \leq & \mu(A_{t+r})+t+r.
\end{eqnarray*}
Similarly, we have
\begin{eqnarray*}
\mu(A) & \leq & (1-t)\mu(A) + t\mu(A) \leq (1-t)m(A_r)+r+t\\
& \leq & t\nu(A_{t+r}) + (1-t) m(A_{t+r})+t+r.
\end{eqnarray*}
By the above inequalities we get that $\rho(t\nu+(1-t)m,\mu)\leq t+r$.
\end{proof}
The crucial property of $T_t$ is that they are open. This result is due to L. Q. Eifler
\cite{E1}. For the sake of completeness we include a proof.
\begin{prop} \label{p2}
For every $t\in (0,1)$, the function $T_t$ is open.
\end{prop}
\begin{proof}
Let us introduce some notation. given $\mu,\nu\in M_+(X)$ we write $\nu\leq \mu$
if $\nu(A)\leq \mu(A)$ for every $A\subseteq X$ Borel (note that $\nu\leq \mu$ implies
that $\nu\ll\mu$). Moreover, for every $\mu\in M_+(X)$ and $f\in C_b(X)$ by $f\mu$
we denote the measure defined by $f\mu(A)=\int_a f(x)d\mu(x)$.
\medskip

\noindent \textsc{Claim.} \textit{Let $\mu,\nu\in M_+(X)$ with $\nu\leq\mu$ and $(\mu_n)$
in $M_+(X)$ with $\mu_n\to\mu$. Then there exist a sequence $(\nu_k)$ in $M_+(X)$ and a
subsequence $(\mu_{n_k})$ of $(\mu_n)$ such that $\nu_k\to \nu$ and $\nu_k\leq \mu_{n_k}$
for every $k$.}
\medskip

\noindent \textit{Proof.} Let $g=d\nu/d\mu\in L^1(\mu)$. As $\nu\leq\mu$ we have that
$0\leq g(x)\leq 1$ $\mu$-a.e. So we may find a sequence $(f_k)$ in $C_b(X)$ such that
$f_k\to g$ in $L^1(\mu)$ and $0\leq f_k(x)\leq 1$ for every $x\in X$ and every $k$.
It follows that $f_k\mu\to g\mu=\nu$. Moreover, for fixed $k$, we have $f_k\mu_n\to f_k\mu$
as $n\to\infty$. Pick a subsequence $(\mu_{n_k})$ of $(\mu_n)$ such that
$f_k\mu_{n_k}\to\nu$ and set $\nu_k=f_k\mu_{n_k}$. \hfill $\lozenge$
\medskip

\noindent Fix $0<t<1$. Let $w=t\mu+(1-t)\nu$, where $\mu,\nu\in P(X)$. It is enough to
show that for every sequence $(w_n)$ in $P(X)$ with $w_n\to w$, there exists a
subsequence $(w_{n_k})$ of $(w_n)$ and sequences $(\mu_k)$, $(\nu_k)$ in $P(X)$
such that $\mu_k\to\mu$, $nu_k\to\nu$ and $w_{n_k}=t\mu_k+(1-t)\nu_k$. So let
$(w_n)$ be one. As $t\mu\leq w$, by the above claim, there exist a sequence
$(m_k)$ in $M_+(X)$ and a subsequence $(w_{n_k})$ of $(w_n)$ such that $m_k\to t\mu$
and $m_k\leq w_{n_k}$. Set $t_k=m_k(X)$. Then $t_k\to t$. By passing to further
subsequences if necessary, we may assume that the sequence $(t_k)$ is monotone.
Without loss of generality suppose that $t_k\uparrow t$ (the case $t_k\downarrow t$
is similar). Set $m'_k=m_k/t_k$ and $\nu_k=(w_{n_k}-m_k)/(1-t_k)$. Then
$m'_k,\nu_k\in P(X)$ and
\begin{eqnarray*}
w_{n_k} & = & t_k m'_k +(1-t_k)\nu_k = (t-t-t_k)m'_k+ (1-t+t-t_k) \nu_k\\
& = & t\Big(m'_k - \frac{t-t_k}{t} m'_k+ \frac{t-t_k}{t} \nu_k\Big) +(1-t)\nu_k.
\end{eqnarray*}
Put
\[ \mu_k= m'_k- \frac{t-t_k}{t} m'_k + \frac{t-t_k}{t} \nu_k.\]
Then $\mu_k\in P(X)$. As $\mu_k\to\mu$ and $\nu_k\to\nu$ the proof is completed.
\end{proof}
\begin{rem} \label{r1}
Note that the above proof is valid for all metrizable spaces, as long as we restrict
ourselves to Radon probability measures. In this direction the more general result
is contained in \cite{E2}.
\end{rem}
Finally, we need the following elementary property of continuous, open functions.
\begin{lem} \label{l3}
Let $Z, Y$ be Baire spaces and $f:Z\to Y$ a continuous, open function. Then
$f^{-1}(G)$ is co-meager in $Z$ for every $G\subseteq Y$ co-meager.
\end{lem}
We are ready to give our main result of this section.
\begin{thm} \label{t4}
Let $X$ be a Polish space and $F$ a face of $P(X)$ with the Baire property. Then
$F$ is either a meager or a co-meager subset of $P(X)$.
\end{thm}
\begin{proof}
Assume that $F$ is not meager. We will show that in this case $F$ is actually co-meager
and this will finish the proof. As $F$ is not meager, there exist $\mu\in P(X)$ and
$\lambda>0$ such that $F$ is co-meager in $B(\mu,\lambda)$. Pick $a,r>0$ small enough
such that $a+r<\lambda/2$. Set $Z=P(X)\times B(\mu,r)$ and $Y=B(\mu,\lambda)$.
Consider the restriction of $T_a$ on $Z$. Denote the restriction by $f$. Clearly $f$
remains open and continuous. By Lemma \ref{l1} we have that $f(Z)\subseteq Y$. Put
$W=F\cap Y$. By Lemma \ref{l3}, we get that $f^{-1}(W)$ is co-meager in $Z$.

Set $G=F\cap B(\mu,r)$ (note that $G$ is co-meager in $(B(\mu,r)$). As $F$ is an
extreme convex subset of $P(X)$, for every $\nu\in P(X)$ and every $m\in B(\mu,r)$
we have that $T_a(\nu,m)\in F$ if and only if $\nu,m\in F$. It follows that
\[ f^{-1}(W)=F\times G,\]
whence $F$ must be co-meager and the proof is completed.
\end{proof}

%------------------------------------------------------------------%

\section{The test measures of an analytic Haar-null set}

Let $X$ be an abelian Polish group (locally or non-locally compact). For every
universally measurable set $A\subseteq X$ we put
\[ T(A)=\{\mu\in P(X):\mu(x+A)=0 \text{ for every } x\in X\}.\]
Namely $T(A)$ is the set of test measures of $A$. We have the following estimate
for the complexity of $T(A)$ (see \cite{D}, Lemma 4).
\begin{lem} \label{l5}
 If $A\subseteq X$ is analytic, then $T(A)$ is a co-analytic subset of $P(X)$.
\end{lem}
Note the for every universally measurable set $A$, the set $T(A)$ is a face of $P(X)$.
Hence by Theorem \ref{t4} and Lemma \ref{l5}, we get the following topological dichotomy
for the set of test measures of an analytic Haar-null set.
\begin{cor} \label{c6}
Let $X$ be an abelian Polish group. Then for every analytic Haar-null set $A\subseteq X$,
the set $T(A)$ is either a meager or a co-meager subset of $P(X)$.
\end{cor}
The following definition is motivated by Corollary \ref{c6}.
\begin{defn} \label{d7}
An analytic Haar-null set $A\subseteq X$ is called strongly Haar-null if $T(A)$ is
co-meager. Otherwise, it is called \textit{weakly Haar-null}.
\end{defn}
Under the terminology of the above definition, Corollary \ref{c6} states that every
analytic weakly Haar-null set is necessarily tested by meager many measures. We
should point out, however, that even if $T(A)$ may be small (in the topological sense),
it is always dense in $P(X)$. Indeed, as has been indicated by Christensen (see \cite{C}),
for every $x\in X$ and $r>0$ there exists an $\mu\in T(A)$ such that
$\rho(\mu,\delta_x)<r$. But the set of convex combinations of Dirac measures is dense
in $P(X)$. So, using the properties of the L\'{e}vy metric, we can easily
verify the density of $T(A)$.

As has been shown in \cite{D}, every analytic and non-meager Haar-null set is necessarily
weakly Haar-null (and so strongly Haar-null sets are necessarily meager). On the other
hand, we have the following proposition (for a proof see \cite{D}, Proposition 5).
\begin{prop} \label{p8}
Let $X$ be an abelian Polish group and $A\subseteq X$ a $\sigma$-compact Haar-null
set. Then $A$ is strongly Haar-null.
\end{prop}
By the above proposition, it follows that in locally-compact groups closed Haar-null
sets are strongly Haar-null.

The situation in non-locally-compact groups is quite different. In particular, we will
show that in every non-locally-compact abelian Polish group $X$, there exists a closed
weakly Haar-null set $F\subseteq X$. To this end we mention that the collection $F(X)$
of all closed subsets of $X$ equipped with the Effros-Borel structure is a standard
Borel space (see \cite{Kechris}, page 75). Denote by $\mathcal{H}$ the collection
of all closed Haar-null sets and by $\mathcal{H}_s$ the collection of all closed
strongly Haar-null sets. Clearly $\mathcal{H}\supseteq \mathcal{H}_s$. We claim that
the inclusion is strict. Indeed, observe that, as noted by Solecki in \cite{S2}
(page 211), the set
\[ H=\big\{ (F,\mu)\in F(X)\times P(X): \mu(x+F)=0 \ \forall x\in X\big\}\]
is $\PB^1_1$. It follows that the set
\[ \mathcal{H}_s=\big\{F\in F(X): \{\mu\in P(X):(F,\mu)\in H\} \text{ is co-meager}\big\}\]
is $\PB^1_1$ too (see \cite{Kechris}, page 244). But, as has been proved by Solecki,
$\mathcal{H}$ is $\SB^1_1$-hard, hence not $\PB^1_1$ (for the definition of $\SB^1_1$-hard
sets see \cite{Kechris} or \cite{S2}). So the inclusion is strict.

Summarizing, we get the following corollary which provides another characterization
of non-locally-compact abelian Polish groups via properties of Haar-null sets.
\begin{cor} \label{c9}
Let $X$ be an abelian Polish group. Then $X$ is non-locally-compact if and only if
there exists a closed weakly Haar-null set $F\subseteq X$.
\end{cor}
\begin{rem} \label{r2}
Clearly the crucial property of $T(A)$ we have used in order to prove Corollary \ref{c6},
is that $T(A)$ has the Baire property. If $A\subseteq X$ is co-analytic, then we can easily
verify that $T(A)$ is $\PB^1_2$. Under the standard axioms of set theorem, $\PB^1_2$
do not necessarily have the property of Baire. However, under any other additional
axiom, capable of establishing that $\PB^1_2$ have the Baire property (such as
$\SB^1_1$-determinacy or Martin's Axiom and the negation of Continuum Hypothesis),
Corollary \ref{c6} is valid for co-analytic sets. Moreover, it can be easily checked
that under Projective Determinacy the dichotomy is valid for every projective set.
\end{rem}

%----------------------------------------------------------------------%

\section{Sets invariant under $K_\sigma$ subgroups}

In \cite{MZ}, E. Matou\v{s}kov\'{a} and M. Zelen\'{y} (using methods introduced by
S. Solecki in \cite{S1}) proved the following.
\begin{thm} \label{t10}
Let $X$ be a non-locally-compact abelian Polish group. Then there exist closed
non-Haar-null sets $A, B\subseteq X$ such that the set $(x+A)\cap B$ is compact
for every $x\in X$.
\end{thm}
In what follows, by $A$ and $B$ we denote the sets obtained from Theorem \ref{t10}.
We will use the following notation. For any set $S\subseteq X$ we put
\[ I(S)=\bigcup_{x\in S} \big( (x+A)\cap B\big).\]
We have the following fact.
\begin{lem} \label{l11}
If $K\subseteq X$ is compact, then $I(K)$ is compact.
\end{lem}
\begin{proof}
First observe that $I(K)$ is closed. Indeed, note that $I(K)=(K+A)\cap B$. Then,
as $K$ is compact and $A$ is closed, we get that $K+A$ is closed, whence so is
$I(K)=(K+A)\cap B$.

So it is enough to show that $I(K)$ is totally bounded. We will need certain
facts from the construction of $A$ and $B$ made in \cite{MZ}. By $d$ we denote
a compatible complete translation invariant metric of $X$.

Fix $S=(s_n)$ a dense countable subset of $X$ and $(Q_k)$ an increasing sequence
of finite subsets of $X$ such that $\bigcup_k Q_k$ is dense in $X$. Fix also sequences
$(\ee_m)$ and $(\delta_m)$ or real numbers such that $\sum_{i>m} \ee_i<\delta_m/2$
and that for every $m$ the satisfy the following fact proved in \cite{MZ}:
\textit{For any finite sets $F_1$ and $F_2$, there exists $g\in B(0,\ee)$
such that $\mathrm{dist}(F_1,g+F_2)\geq \delta$}. Note that $\delta_m\to 0$.

Matou\v{s}kov\'{a} and Zelen\'{y}, by induction, constructed sequences $(g^m_k)$
and $(\tilde{g}^m_k)$ which, among other things, satisfy the following
(crucial for our considerations) property.
\[ (P) \ \ \ \forall n \text{ and } \forall i,j \text{ with } i,j\geq n
\text{ we have } \mathrm{dist}(s_n+g^m_i+Q_i, \tilde{g}^m_j+Q_j) \geq 3\delta_m. \]
Then they defined
\[ A=\bigcap_{m\geq 1} \bigcup_{k\geq 1} \overline{B(g^m_k+Q_k, \delta_m)} \ \text{ and } \ B=\bigcap_{m\geq 1} \bigcup_{k\geq 1} \overline{B(\tilde{g}^m_k+Q_k, \delta_m)}. \]

Let $r>0$ arbitrary. Pick $m$ such that $\delta_m<r$. From the compactness of $K$,
pick $l$ and $(s_{n_i})_{i=1}^l\subset S$ such that
\[ \bigcup_{i=1}^l B(s_{n_i}, \delta_m/2)\supseteq K. \]
If $x\in X$ and $s_n\in S$ with $d(s_n,x)<\delta_m/2$, then, as shown in \cite{MZ},
from property $(P)$ we have
\[ (x+A)\cap B\subset \Big( \bigcup_{k=1}^{n-1} \overline{B(s_n+g^m_k+Q_k, 3\delta_m/2)}
\Big) \cup \Big( \bigcup_{k=1}^{n-1} \overline{B(\tilde{g}^m_k+Q_k,\delta_m)}\Big).\]
It follows that
\[ I(K)\subseteq \bigcup_{i=1}^l \Big( \Big( \bigcup_{k=1}^{n-1}
\overline{B(s_{n_i}+g^m_k+Q_k, 3\delta_m/2)} \Big) \cup \Big(
\bigcup_{k=1}^{n-1} \overline{B(\tilde{g}^m_k+Q_k,\delta_m)}\Big) \Big) .\]
So the set $I(K)$ can be covered by finitely many balls of radii $2\delta_m<2r$. As $r$
was arbitrary, we get that $I(K)$ is totally bounded and the proof is completed.
\end{proof}
\begin{prop} \label{p12}
Let $X$ be non-locally-compact abelian Polish group and $G$ a $\sigma$-compact subgroup
of $X$. Then there exists a $G$-invariant $F_\sigma$ subset $F$ of $X$ such that $F$
is neither prevalent nor Haar-null. In particular, this holds if $G$ is a countable
dense subgroup of $X$.
\end{prop}
\begin{proof}
Put $G=\bigcup_n K_n$, where each $K_n$ is compact. Let $A$ and $B$ be as in Theorem
\ref{t10}. Define
\[ C_1=\bigcup_{x\in G} (x+A) \ \text{ and } \ C_2=\bigcup_{x\in G} (x+B).\]
Note that for every $n$ the sets $K_n+A$ and $K_n+B$ are closed. Also observe that
$C_1=\bigcup_n K_n +A$ and $C_2=\bigcup_n K_n+B$. So $C_1$ and $C_2$ are $F_\sigma$.
In addition, from the fact that $G$ is a subgroup of $X$, we get that both
$C_1$ and $C_2$ are $G$-invariant.

As any possible translate of a non-Haar-null set is non-Haar-null, we have that $C_1$
and $C_2$ are non-Haar-null. We claim that at least one of them is not prevalent.
Indeed, suppose that both $C_1$ and $C_2$ were prevalent. So $C_1\cap C_2$ would
be prevalent too. But observe that
\begin{eqnarray*}
C_1\cap C_2 & = & \Big( \bigcup_{x\in G} (x+A)\Big) \cap
\Big( \bigcup_{y\in G} (y+B)\Big)\\
& = & \bigcup_{n,m} \bigcup_{x\in K_n} \bigcup_{y\in K_m} (x+A)\cap (y+B)\\
& = & \bigcup_{n,m} \bigcup_{x\in K_n} \bigcup_{y\in K_m} y+\big( (x-y+A)\cap B\big)\\
& \subset & \bigcup_{n,m} K_m + I(K_n-K_m).
\end{eqnarray*}
Clearly the set $K_n-K_m$ is compact. By Lemma \ref{l11}, the set $I(K_n-K_m)$
is compact too, whence so is the set $K_m+I(K_n-K_m)$. But in non-locally-compact
groups, compact sets are Haar-null. It follows that $C_1\cap C_2$ is Haar-null
and we derive a contradiction. Hence at least one of them is not prevalent,
as claimed.
\end{proof}
\begin{rem} \label{r3}
Note that if $G$ is countable, then the proof of Proposition \ref{p12} is considerably
simpler (in particular, Lemma \ref{l11} is completely superfluous). However,
Proposition \ref{p12} shows that invariance under bigger subgroups is not sufficient to
establish the desired dichotomy. For instance, if $X$ is a separable Banach space
and $(X_n)$ an increasing sequence of finite-dimensional subspaces with
$\bigcup_n X_n$ dense in $X$, then we may let $G=\bigcup_n X_n$ and still provide
a counterexample for the dichotomy.
\end{rem}

%------------------------------------------------------------------%

\section{On the measure-theoretic structure of $T(A)$}

By Lemma \ref{l5}, for every $A\subseteq X$ analytic, the set $T(A)$ is a universally
measurable subset of $P(X)$. So it is natural to wonder about the measure-theoretic
structure of $T(A)$. If $X$ is locally-compact, then it is easy to see that there exists
an $M\in P(P(X))$ (namely $M$ is a measure on measures) such that $M(T(A))=1$
for every $A\subseteq X$ universally measurable. Indeed, let $h$ be the Haar measure
and $H$ be the subset of $P(X)$ consisting of all probability measures $\mu$ absolutely
continuous with respect to $h$. So if $M\in P(P(X))$ is any measure on measures
supported in $H$, then $M(T(A))=1$ for every universally measurable
Haar-null set $A\subseteq X$.

In the non-locally-compact case the situation is different. As M. B. Stinchcombe
observe in \cite{St}, for every $M\in P(P(X))$ there exists a $\sigma$-compact
set $K\subseteq X$ such that $M(\{\mu:\mu(K)=0\})=0$. Hence for every $M\in P(P(X))$
there exists a Haar-null set $K\subseteq X$ for which $M(T(A))=0$. Nevertheless,
we will show that in the non-locally-compact case there exists a measure-theoretic
analogue of Corollary \ref{c6}. So in what follows we will assume that $X$ is
non-locally-compact. As before, $d$ is a translation invariant metric of $X$.

Let us introduce some notation. For every $x\in X$ and every $\mu\in P(X)$ define
the probability measure $\mu_x$ by $\mu_x(A)=\mu(A-x)$ for every Borel set
$A\subseteq X$. That is $\mu_x=\mu*\delta_x$ (note that $\supp\mu_x=x+\supp\mu$).
So every $x\in X$ gives rise to the function $g_x:P(X)\to P(X)$ defined by
$g_x(\mu)=\mu_x$. Observe that the family $(g_x)_{x\in X}$ is an uncountable
group of homeomorphisms of $P(X)$ (actually, $G$ is a group of isometries of
$(P(X),\rho)$). Also note that for every $A\subseteq X$ universally measurable,
the set $T(A)$ is $G$-invariant.
\begin{lem} \label{l13}
Let $\mu\in P(X)$ be such that $\supp\mu$ is compact. Then the orbit $G\mu$
of $\mu$ is uncountable.
\end{lem}
\begin{proof}
Assume, on the contrary, that $G\mu$ is countable. Write $G\mu=(\nu_i)$
and for every $i$ let
\[ C_i=\{x\in X: \mu_x=\nu_i\}.\]
It is easy to verify that every $C_i$ is closed. As $X=\bigcup_i C_i$ by
the Baire category Theorem there exists a $k$ such that
$\mathrm{Int}(C_k)\neq\varnothing$. So there exist $z\in X$ and $r>0$ such that
$\mu_x=\mu_y=\nu_k$ for every $x,y\in B(z,r)$. But this implies that
$x+\supp\mu=y+\supp\mu$ for every $x,y\in B(z,r)$ and so
\[ B(0,r)\subseteq B(z,r)-B(z,r)\subseteq \supp\mu-\supp\mu,\]
which is impossible as $\supp\mu$ is compact and $X$ is non-locally-compact.
\end{proof}
If $\mu$ is any probability measure, then the orbit $G\mu$ of $\mu$ is Borel.
To see this, observe that the group $X$ acts continuously on $P(X)$ under
the action $x\cdot \mu=\delta_x*\mu$ and that this action is precisely the action
of $G$. It follows by a result of Miller (see \cite{Kechris} or \cite{Mi})
that the orbit $G\mu$ of $\mu$ is Borel. Note, however, that if $\mu$ is a
Dirac measure, then its orbit is the set of all Dirac measures. But, as is well-known,
this is a closed subset of $P(X)$. This fact is generalized in the following
proposition.
\begin{prop} \label{p14}
Let $\mu\in P(X)$ be such that $\supp\mu$ is compact. Then the orbit $G\mu$
of $\mu$ is closed.
\end{prop}
To prove Proposition \ref{p14} we will need a certain consequence of the classical
Ramsey Theorem \cite{Ra} for doubletons of $\nn$. As is common in Ramsey Theory,
for every $I\subseteq \nn$ by $[I]^2$ we denote the collection of all doubletons
$(i,j)$ such that $i,j\in I$ and $i<j$.
\begin{lem} \label{l15}
Let $r>0$ and $(x_n)$ be a sequence such that $d(x_n,x_m)\geq r$ for every
$n\neq m$. Then for every $K\subseteq X$ compact, there exists a subsequence $(y_n)$
of $(x_n)$ such that
\[ B(y_i-y_j,r/8)\cap K=\varnothing\]
for every $i,j$ with $i<j$.
\end{lem}
\begin{proof}
Let
\[ A=\big\{ (i,j)\in [\nn]^2: B(x_i-x_j,r/8)\cap K=\varnothing\big\}\]
and
\[ B=\big\{ (i,j)\in [\nn]^2: B(x_i-x_j,r/8)\cap K\neq\varnothing\big\}.\]
Then $[\nn]^2=A\cup B$ and $A\cap B=\varnothing$. By Ramsey's Theorem, there exists
$M\subseteq \nn$ infinite such that either $[M]^2\subseteq A$ or $[M]^2\subseteq B$.

We claim that the $[M]^2\subseteq A$. Indeed, assume on the contrary that $[M]^2\subseteq B$.
Set $m=\min\{ i:i\in M\}$ and $M'=M\setminus \{m\}$. Then for every $n\in M'$ we have
\[ B(x_m-x_n,r/8)\cap K\neq \varnothing.\]
So for every $n\in M'$ there exists $w_n\in K$ such that $d(x_m-x_n,w_n)<r/8$. As $K$
is compact, pick a $(z_l)_{l=1}^k$ a finite $r/8$-net of $K$. By cardinality arguments,
we get that there exist an infinite set $I\subseteq M'$ and an $l\in \{1,...,k\}$
such that $d(w_i,z_l)<r/8$ for every $i\in I$. Hence $d(w_i,w_j)<r/4$ for every
$i,j\in I$. But if $i,j\in I$ with $i\neq j$, then we have
\begin{eqnarray*}
d(x_i,x_j) & = & d(-x_i,-x_j)= d(x_m-x_i,x_m-x_j)\\
& \leq & d(x_m-x_i, w_i)+d(w_i,w_j)+ d(w_j,x_m-x_j)< r/2,
\end{eqnarray*}
which contradicts our assumption on $(x_n)$. Hence $[M]^2\subseteq A$.

Now let $(m_n)$ be the increasing enumeration of $M$ and set $y_n=x_{m_n}$ for
every $n$. Clearly $(y_n)$ is the desired sequence.
\end{proof}
We continue with the proof of Proposition \ref{p14}
\begin{proof}[Proof of Proposition \ref{p14}]
First of all observe that we may assume that $0\in\supp\mu$. indeed, if $y\in\supp\mu$,
then set $\nu=\mu_{-y}$ and observe that $0\in\supp\nu$ and that $G\nu=G\mu$. So in what
follows we will assume that $0\in\supp\mu$.

Let $(x_n)$ in $X$ and $\nu\in P(X)$ be such that $g_{x_n}(\mu)=\mu_{x_n}\to\nu$.
Note that it suffices to prove that there exist $x\in X$ and a subsequence $(x_{n_k})$
of $(x_n)$ such that $x_{n_k}\to x$. Indeed, in this case observe that $\delta_{x_{n_k}}
\to \delta_x$. Hence $\mu_{x_{n_k}}=\mu*\delta_{x_{n_k}}\to \mu*\delta_x=\mu_x$.
On the other hand, as $(\mu_{x_{n_k}})$ is a subsequence of $(\mu_{x_n})$ we still
have that $\mu_{x_{n_k}}\to\nu$. This implies that $\nu=\mu_x$ as desired.

Assume, towards a contradiction, that the sequence $(x_n)$ does not contain any
convergent subsequence. So there exist an $r>0$ and a subsequence $(x_{n_k})$
of $(x_n)$ such that $d(x_{n_i}, x_{n_j})\geq r$ for every $i\neq j$.

Applying Lemma \ref{l15} to the sequence $(x_{n_k})$ for $K=\supp\mu$, we get a
subsequence $(y_n)$ of $(x_{n_k})$ (and so of $(x_n)$) such that for every $i,j$
with $j>i$
\[ B(y_i-y_j,r/8)\cap\supp\mu=\varnothing.\]
The subsequence $(\mu_{y_n})$ of $(\mu_{x_n})$, determined by $(y_n)$, still
converges to $\nu$. From the properties of $(y_n)$ we have that if $n>i$, then
\[ \mu_{y_n}\big(B(y_i,r/8)\big)=\mu\big(B(y_i,r/8)-y_n\big)=
\mu\big(B(y_i-y_n,r/8)\big)=0 \]
as $B(y_i-y_n,r/8)\cap\supp\mu=\varnothing$. As for every $U\subseteq X$ open the
function $\mu\to\mu(U)$ is lower semicontinuous, we get
\[ \nu\big(B(y_i,r/8)\big)\leq\liminf \mu_{y_n}\big(B(y_i,r/8)\big)=0\]
for every $i$. Hence the set $V=\bigcup_i B(y_i, r/8)$ is $\nu$-null.

Now set $F=\bigcup_i \overline{B(y_i,r/9)}$. Note that $F$ is closed, as
$d(y_i,y_j)\geq r$ for $i\neq j$, and that $V\supset F$. Then observe that
for every $n$, we have
\[ \mu_{y_n}(F)=\mu(F-y_n)=\mu\Big(\bigcup_i \overline{B(y_i-y_n,r/9)}\Big)
\geq \mu\big(B(0,r/9)\big)>0, \]
where the last inequality holds from the fact that $0\in\supp\mu$. By the upper
semicontinuity of the function $\mu\to\mu(F)$, we get
\[ \nu(F)\geq \limsup \mu_{y_n}(F)\geq \mu\big( B(0,r/9)\big)>0.\]
But this implies that
\[ 0=\nu(V)\geq \nu(F)\geq \mu\big(B(0,r/9)\big)>0\]
and we derive the contradiction.
\end{proof}
\begin{cor} \label{c16}
Let $X$ be a non-locally-compact abelian Polish group and $A\subseteq X$ an analytic
Haar-null set. Then there exists a continuous Borel probability measure $M$ on $P(X)$
such that:
\begin{enumerate}
\item[(i)] $M(T(A))=1$.
\item[(ii)] If $B\subseteq X$ is any other analytic Haar-null set, then either
$M(T(B))=1$ or $M(T(B))=0$.
\end{enumerate}
\end{cor}
\begin{proof}
Pick any compactly supported probability measure $\mu\in T(A)$. By Lemma \ref{l13}
and Proposition \ref{p14}, the orbit $G\mu$ of $\mu$ is an uncountable closed
subset of $P(X)$. So if $M$ is any continuous measure on measures supported in
$G\mu$, then $M(T(A))=1$. Finally, observe that if $B\subseteq X$ is any other analytic
Haar-null set, then as $T(B)$ is $G$-invariant either $G\mu\subseteq T(B)$ or
$G\mu\cap T(B)=\varnothing$. This clearly establishes property (ii).
\end{proof}
\noindent \textbf{Acknowledgments.} I am indebted to Professor S. Papadopoulou for
providing me with references \cite{E1} and \cite{E2}. I would also like to
thank Professor A. S. Kechris for his interest on the paper.

%-------------------------------------------------------------------%
%                           Bibliography                            %
%-------------------------------------------------------------------%

\end{document}